\documentclass[10pt]{article}

\usepackage{geometry, enumerate, amsthm, amsmath, amssymb, amscd, color}
\usepackage{hyperref}
\usepackage{courier}
\usepackage[all]{xy}

\numberwithin{equation}{section}

\newtheorem{Thm}{Theorem}[section]
\newtheorem{Prop}[Thm]{Proposition}
\newtheorem{Lem}[Thm]{Lemma}

\theoremstyle{definition}\newtheorem{Def}[Thm]{Definition}

\newtheorem{Rem}[Thm]{Remark}
\theoremstyle{definition}

\newcommand{\N}{\mathbb{N}}
\newcommand{\Q}{\mathbb{Q}}
\newcommand{\Z}{\mathbb{Z}}
\newcommand{\R}{\mathbb{R}}
\newcommand{\Int}{\textnormal{Int}}

\newcommand{\mcP}{\mathcal{P}}

\newcommand{\Pirr}{\mcP^{{\rm irr}}}

\newcommand{\Gal}{{\rm Gal}}

\newcommand{\W}{\widehat{W}}
\newcommand{\mcW}{\mathcal{W}}
\newcommand{\K}{\widehat{K}}
\newcommand{\abK}{\overline{\K}}
\newcommand{\V}{\widehat{V}}
\newcommand{\abV}{\overline{\V}}

\newcommand{\PL}{\mathbb{P}^1}

\begin{document}

\title{Transcendental extensions of a valuation domain of rank one}

\date{11 May 2016}

\author{Giulio Peruginelli}


\leftmark{\noindent  accepted for publication in the Proc. Amer. Math. Soc. (2016).}

{\let\newpage\relax\maketitle} 

\begin{abstract}
 Let $V$ be a valuation domain of rank one and quotient field $K$. Let $\abK$ be a fixed algebraic closure of the $v$-adic completion $\K$ of $K$ and let $\abV$ be the integral closure of $\V$ in $\abK$. We describe a relevant class of valuation domains $W$ of the field of rational functions $K(X)$ which lie over $V$, which are indexed by the elements  $\alpha\in\abK\cup\{\infty\}$, namely, the valuation domains $W=W_{\alpha}=\{\varphi\in K(X) \mid \varphi(\alpha)\in\abV\}$. If $V$ is discrete and $\pi\in V$ is a uniformizer, then a valuation domain $W$ of $K(X)$ is of this form if and only if the residue field degree $[W/M:V/P]$ is finite and $\pi W=M^e$, for some $e\geq 1$, where $M$ is the maximal ideal of $W$. In general, for $\alpha,\beta\in\abK$ we have $W_{\alpha}=W_{\beta}$ if and only if $\alpha$ and $\beta$ are conjugated over $\K$. Finally, we show that the set $\Pirr$ of irreducible polynomials over $\K$ endowed with an ultrametric distance introduced by Krasner is homeomorphic to the space  $\{W_{\alpha} \mid \alpha\in\abK\}$ endowed  with the Zariski topology.
\end{abstract}
\medskip
\noindent{\small \textbf{Keywords}: Discrete valuation domain, Valuation overrings, Integer-valued polynomial, Pr\"ufer domain, Zariski topology, Ultrametric space.}


\section{Introduction}

Let $V$ be a valuation domain of rank 1, quotient field $K$ and let $v$ be the associated valuation. Let $\V$ and $\K$ be the $v$-adic completions of $V$ and $K$, respectively. Given a field extension $K\subset F$ and a valuation domain $W$ of $F$, we say that 
$W$ lies above $V$ if $W\cap K=V$. If $F=K(\theta)$ is a simple separable algebraic extension of $K$ and $p\in K[X]$ is the minimal polynomial of $\theta$, then the valuation domains $W$ of $F$ which lie above $V$ are well-known: they are rank $1$ valuation domains which are in one-to-one correspondence with the irreducible factors over $\K$ of $p(X)$ (see for example \cite[Chapter VI, \S. 8, 2., Proposition 2 and Corollaire 2]{Bourb} or \cite[Chapter 6, B.]{Rib}). More precisely, there exists a finite set of elements $\{\theta_1,\ldots,\theta_n\}$ in a fixed algebraic closure $\abK$ of $\K$ (i.e., the roots of $p(X)$ in $\abK$) such that the above valuation domains $W$ are equal to $W_{\theta_i}=\{g(\theta)\in K(\theta) \mid g(\theta_i) \textnormal{ is integral over }\V\}$, for $i=1,\ldots,n$; moreover, $W_{\theta_i}=W_{\theta_j}$ if and only if $\theta_i,\theta_j$ are conjugated over $\K$ (i.e.,   $\theta_i,\theta_j$ are roots of the same irreducible factor of $p(X)$ over $\K$).

If instead we consider a simple transcendental extension $K(X)$ of $K$, the structure of the set of valuation domains of $K(X)$ which lie above $V$ is much richer (see for example  \cite{APZ, Kuhlmann,LopTart,MacLane,MacLaneDuke,
NartCo, Vaquie}). To begin with, it is well-known that the rank of $W$ is $1$ or $2$ (see for example \cite[Chapt. VI, \S 10, Corollaire 1, p. 162]{Bourb}). The aim of this paper is to give an explicit description of a particular class of these valuation domains, which, likewise the previous algebraic case, arise from the elements $\alpha\in\abK\cup\{\infty\}$, namely $W=W_{\alpha}=\{\varphi\in K(X) \mid \varphi(\alpha)\textnormal{ is integral over } \V\}$.  The description of these valuation domains is accomplished in Proposition \ref{Vpalpha}. When $V$ is a discrete valuation domain of rank one, we give in Theorem \ref{description W} sufficient conditions on a valuation domain $W$ of $K(X)$ lying over $V$ to be of the form $W_{\alpha}$, for some $\alpha\in\abK\cup\{\infty\}$, namely:
\begin{itemize}
\item[i)] the residue field degree $[W/M:V/P]$ is finite;
\item[ii)] $\pi W=M^e$, for some $e\geq 1$;
\end{itemize}
where $\pi$ is a uniformizer of $V$ and $M$ is the maximal ideal of $W$. 

In contrast with the finite algebraic case recalled above, where the extensions of $V$ to $K(\theta)$ are given by a finite set of elements in $\abK$, in the transcendental case one has to consider all the uncountably many elements of $\abK$ in order to obtain all the valuation domains $W\subset K(X)$ of the above form. As in the algebraic case, for $\alpha,\alpha'\in\abK$,  $W_{\alpha}=W_{\alpha'}$ if and only if $\alpha$ and $\alpha'$ are conjugate over $\K$ (Theorem \ref{equivalent conditions}). Moreover, it turns out that these valuation domains are precisely the unitary valuation overrings of a class of generalized integer-valued polynomial rings which was introduced in \cite{LopWer}: given a finite field extension $F$ of $K$, let $V_F$ be the integral closure of $V$ in $F$. We set
\begin{equation}\label{IntQOK}
\Int_{K}(V_F)=\{f\in K[X] \mid f(V_F)\subseteq V_F\}
\end{equation}
Note that $\Int_{K}(V_F)$ is the contraction to $K[X]$ of $\Int(V_F)=\{f\in F[X] \mid f(V_F)\subseteq V_F\}$, the classical ring  of integer-valued polynomials over $V_F$. Given a valuation domain $W\subset K(X)$ as above, we show that there exists a finite extension $F$ of $K$ such that $W$ is an overring of $\Int_K(V_F)$. 

In the second section we give the characterization of the aforementioned valuation domains $W$ of $K(X)$. We show how the valuation domains $W_{\alpha}$, for $\alpha\in\K$, are related to the work of Kaplansky about immediate extensions of a valued field in \cite{Kaplansky}, see Remark \ref{Kap}. More generally, when $\alpha$ ranges in $\abK$, we show the connection with the work of MacLane in \cite{MacLane, MacLaneDuke} about approximations of transcendental extensions of a DVR, see Remark \ref{MacLane valuations}. The valuation domains $W_{\alpha}$, $\alpha\in\abK$, appear also in the recent paper \cite{NartCo}, which deals with extensions of a DVR to a transcendental extension of its field of quotients in the spirit of MacLane. Furthermore, as an application, we show that the Pr\"ufer domains of polynomials between $\Z[X]$ and $\Q[X]$ constructed in \cite{LopWer} can be represented as rings of integer-valued polynomials (Remark \ref{Unitary val dom LopWer}).  In the third section, for a general rank one valuation domain $V$, we show that the set $\mathfrak{W}=\{W_{\alpha} \mid \alpha\in \abK\}$ is in one-to-one correspondence with the following sets: the set $\Pirr$ of irreducible polynomials over $\K$; the set $\mathfrak{W}_{K[X]}=\{W_{\alpha}\cap K[X] \mid \alpha\in \abK\}$ (Theorem \ref{equivalent conditions}). In particular, this allows us to reduce many considerations to polynomials rather than to rational functions.  Moreover, if we endow $\Pirr$ with the ultrametric distance $\Delta(p,q)=\min\{|\alpha-\beta| \mid \alpha,\beta\in\abK, p(\alpha)=q(\beta)=0\}$ and $\mathfrak{W}$ and $\mathfrak{W}_{K[X]}$ with the Zariski topology, these three spaces are homeomorphic (Theorem \ref{homeomorphism}).  A first evidence of this result is contained in a paper of Gilmer, Heinzer, Lantz and Smith, where, for a DVR $V$ with finite residue field, it is proved that the unitary maximal spectrum of the ring $\Int(V)$ (that is, those maximal ideals whose contraction to $V$ is equal to the maximal ideal of $V$) is homeomorphic to $\V$ (see \cite[p. 677]{GHLS}). Since $\V$ is homeomorphic to the space of monic linear  polynomials endowed with the above distance $\Delta(\cdot,\cdot)$, Theorem \ref{homeomorphism} is a  generalization of this result.

\section{Main result}

Throughout the paper, we adopt the following notation. Let $V$ be a valuation domain of rank $1$ with quotient field $K$. We denote by $v$ the associated valuation on $K$ and by $\widehat P,\V,\K$ the $v$-adic completion of $P,V,K$, respectively. Note that $\V\cap K=V$, $\V$ is a valuation domain of rank $1$ of $\K$ with residue field isomorphic to the residue field of $V$ and $v$ extends uniquely to the valuation of $\K$ associated to $\V$, which we still denote by $v$. Let $\abK$ be a fixed algebraic closure of $\K$ and $\abV$ the integral closure of $\V$ in $\abK$. It is well known that $v$ admits a unique extension to $\abK$ (\cite[Chapt. 5, A.]{Rib}), which again we denote by $v$ and whose valuation ring is $\abV$, thus $\abV=\{\alpha\in\abK \mid v(\alpha)\geq 0\}$. For $\alpha\in\abK$, we denote by $\V_{\alpha}$ the valuation domain of rank $1$ of the finite field extension $\K(\alpha)$ of $\K$, which is equal to the integral closure of $\V$ in $\K(\alpha)$, and by $\widehat{P}_{\alpha}$ the maximal ideal of $\V_{\alpha}$. We denote by $v_{\alpha}$ the valuation of $\K(\alpha)$ associated to $\V_{\alpha}$ (thus, the restriction of $v$ to $\K(\alpha)$). 

We recall the notion of Gaussian extension of $v$. Given $f(X)=\sum_{i\geq0}^n a_i X^i\in K[X]$, we set $v_G(f)\doteqdot\min\{v(a_i)\mid i=0,\ldots,n\}$, and this function extends in the natural way to a valuation of $K(X)$, called the Gaussian extension of $v$. The valuation domain of the Gaussian extension is equal to $V[X]_{P[X]}$ (see for example \cite[Proposition 18.7]{Gilm} or \cite[Chapt. VI, \S 10]{Bourb}) and is the unique extension of $V$ to $K(X)$ such that $X$ is transcendental over the residue field  (\cite[Chapt. VI, \S 10, Prop. 2]{Bourb}). 

Given a field extension $K\subset F$, a valuation domain $W$ of $F$ is immediate over $K$ if the value groups and the residue fields of $W$ and $W\cap K$ are the same, respectively. For $\alpha\in\abK$, note that the elements of the value group of $\V_{\alpha}\subset \K(\alpha)$ are of the form $v_{\alpha}(g(\alpha))$, where $g\in\K[X]$. If $h\in K[X]$ is such that $v_G(h-g)$ is sufficiently greater than $v_{\alpha}(g(\alpha))$, then 
$v_{\alpha}(g(\alpha)-h(\alpha))>v_{\alpha}(g(\alpha))$, so that $\V_{\alpha}$ is immediate over $K(\alpha)$ (see \cite[\S10, Exercise 2, p. 193]{Bourb}). 

In order to describe all the possible valuation domains of $K(X)$ we are interested in, we need to consider the projective line over $\abK$, that is, $\PL(\abK)\doteqdot\abK\cup\{\infty\}$. Given a rational function $\varphi\in K(X)$ and $\alpha\in \abK$, $\varphi(\alpha)$ is an element of $\PL(\abK)$; we say that $\varphi$ is not defined at $\alpha$ if $\varphi(\alpha)=\infty$. We also set $\varphi(\infty)\doteqdot \psi(0)$, where $\psi(X)=\varphi(1/X)$, so that each rational function on $K$ determines a map from $\PL(\abK)$ to itself, which is continuous with respect to the $v$-adic topology (see also Remark \ref{continuity rat fun}).

We introduce now the following definition.
\vskip0.3cm
\begin{Def}\label{Def Vpalpha}
Let $\alpha\in\PL(\abK)$. We consider the set of rational functions $\varphi(X)$ over $K$ which are defined at $\alpha$ and such that their evaluation at $\alpha$ is integral over $\V$:
$$W_{\alpha}\doteqdot\{\varphi\in K(X) \mid \varphi(\alpha)\in\abV\}.$$
\end{Def}
\noindent Clearly, for $\alpha\in\PL(\abK)$, $\varphi\in W_{\alpha}\Leftrightarrow \varphi(\alpha)\in \V_{\alpha}\Leftrightarrow v(\varphi(\alpha))\geq0$, where by convention we set $\V_{\infty}=\V$ and $\widehat{P}_{\infty}=\widehat{P}$. We also set $\K(\infty)\doteqdot\K$.  The following proposition characterizes $W_{\alpha}$. It is a standard result, but for the sake of the reader we give a proof here.

\begin{Prop}\label{Vpalpha}
Let  $\alpha\in\PL(\abK)$. Then $W_{\alpha}$ is a valuation domain of $K(X)$ which lies over $V$ with maximal ideal $M_{\alpha}=\{\varphi\in K(X) \mid v(\varphi(\alpha))>0\}$ and rank $1$ or $2$. The rank of $W_{\alpha}$ is $1$ if and only if $\alpha$ is in $\abK$ and is transcendental over $K$, it is $2$ if either $\alpha\in\abK$ is algebraic over $K$ or $\alpha=\infty$. If the rank of $W_{\alpha}$ is $2$, $\alpha\in\abK$ and $q\in K[X]$ is the minimal polynomial of $\alpha$ over $K$, then the DVR $K[X]_{(q)}$ is the valuation overring of $W_{\alpha}$. If $\alpha=\infty$, then $K[\frac{1}{X}]_{(\frac{1}{X})}$ is the valuation overring of $W_{\infty}$. Moreover, the residue field of $W_{\alpha}$ is isomorphic to the residue field of $\V_{\alpha}$ and the value group of $W_{\alpha}/Q_{\alpha}$ is also isomorphic to the value group of $\V_{\alpha}$, where $Q_{\alpha}=0$ if $W_{\alpha}$ has rank $1$ and $Q_{\alpha}$ is the height one prime of $W_{\alpha}$ if $W_{\alpha}$ has rank $2$.
\end{Prop}
\begin{proof}
Let $\alpha$ be any given element of $\PL(\abK)$. 
It is straightforward to show that $W_{\alpha}$ is a valuation domain of $K(X)$ which lies over $V$ and with maximal ideal $M_{\alpha}=\{\varphi\in K(X) \mid v(\varphi(\alpha))>0\}$. In fact, given a rational function $\varphi\in K(X)$ defined at $\alpha$, that is, $\varphi(\alpha)\in\K(\alpha)$, either $\varphi(\alpha)\in\V_{\alpha}$ or $\varphi(\alpha)^{-1}\in\V_{\alpha}$, since $\V_{\alpha}$ is a valuation domain of $\K(\alpha)$. Both of these conditions hold if and only if $\varphi\in W_{\alpha}\setminus M_{\alpha}$, so the latter is the multiplicative group of units of the valuation domain $W_{\alpha}$ (and so $M_{\alpha}$ is its maximal ideal). We remark that $W_{\alpha}$ can be realized as the pullback of the valuation domain of rank $1$ $\V_{\alpha}$ of $\K(\alpha)$, via the evaluation homomorphism ${\rm ev}_{\alpha}$ at $\alpha$:
\begin{align}\label{evaluation map}
{\rm ev}_{\alpha}:K(X)&\to\K(\alpha)\cup\{\infty\}\nonumber\\
\varphi(X)&\mapsto 
\left\{
\begin{array}{cl}
\varphi(\alpha),&\textnormal{ if }\varphi(X)\textnormal{ is defined at }\alpha\\
\infty,&\textnormal{ otherwise}
\end{array}
\right.
\end{align}
The image of ${\rm ev}_{\alpha}$ in (\ref{evaluation map}) is contained in $\K(\alpha)$ if and only if $\alpha$ is transcendental over $K$ if and only if the kernel $Q_{\alpha}=\{\varphi\in W_{\alpha} \mid \varphi(\alpha)=0\}$ of the restriction $({\rm ev}_{\alpha})_{|_{W_{\alpha}}}:W_{\alpha}\to \V_{\alpha}$ is  equal to $(0)$. If $\alpha$ is algebraic over $K$, then the domain of definition of ${\rm ev}_{\alpha}$ is the DVR $K[X]_{(q)}$, where $q\in K[X]$ is the minimal polynomial of $\alpha$ over $K$, so ${\rm ev}_{\alpha}(K[X]_{(q)})\cong K[X]/(q)\cong K(\alpha)$. Moreover, since $Q_{\alpha}\cap V=(0)$, the height one prime ideal $Q_{\alpha}$ of $W_{\alpha}$ is equal to $q(X)K[X]_{(q)}$ in the algebraic case and the one-dimensional valuation overring of $W_{\alpha}$ is equal to  $K[X]_{(q)}$. 

Suppose $\alpha\in\abK$. If $\alpha$ is transcendental over $K$, then $({\rm ev}_{\alpha})_{|_{W_{\alpha}}}W_{\alpha}\to \V_{\alpha}$ is an injective homomorphism that extends to an homomorphism of quotient fields ${\rm ev}_{\alpha}:K(X)\to \K(\alpha)$. Thus $v_{\alpha}\circ{\rm ev}_{\alpha}$ defines a valuation of real rank $1$ on $K(X)$ whose associated valuation ring is $W_{\alpha}$, hence $W_{\alpha}$ is a valuation domain of rank $1$ of $K(X)$. Reciprocally, if $W_{\alpha}$ is a valuation domain of rank $1$ of $K(X)$ and $Q_{\alpha}\not=0$ (the kernel of $({\rm ev}_{\alpha})_{|_{W_{\alpha}}}$), then $Q_{\alpha}=M_{\alpha}$ and $Q_{\alpha}\cap V\not=(0)$, which is a contradiction. Therefore, $Q_{\alpha}=0$ and $\alpha$ is transcendental over $K$. Hence, for $\alpha\in\abK$, $W_{\alpha}$ is a valuation domain of rank $1$ if and only if $\alpha$ is transcendental  over $K$ and it is a valuation domain of rank $2$ otherwise.

We prove now the last claims. Since the image of $W_{\alpha}$ via ${\rm ev}_{\alpha}$ is equal to the valuation domain $\V_{\alpha}\cap K(\alpha)$ and $\V_{\alpha}$ is immediate over $K(\alpha)$, it follows that the residue field of $\V_{\alpha}$ is isomorphic to the residue field of $W_{\alpha}$. Moreover, since the kernel of $({\rm ev}_{\alpha})_{|_{W_{\alpha}}}$ is $Q_{\alpha}$, then also the value group of $\V_{\alpha}$ is isomorphic to the value group of $W_{\alpha}/Q_{\alpha}$.

The case $\alpha=\infty$ is treated by considering the change of variable $K(X)\to K(Y)$, $X\mapsto Y\doteqdot\frac{1}{X}$, so $W_{\infty}\subset K(X)$ is easily seen to correspond to $W_{0}'$ in $K(Y)$. 
\end{proof}
\vskip0.2cm

For the rest of this section, we adopt the following assumptions and notations: let $V$ be a discrete valuation domain of rank $1$ (DVR) and let $\pi\in V$ be a uniformizer of $V$, that is, $\pi$ is a generator of the maximal ideal $P$ of $V$. Note that $\V$ and $\V_{\alpha}$ are also DVRs, $\abV$ is a non-discrete valuation domain of rank $1$ and that $\pi$ is also a uniformizer of $\V$. 

On the other hand, recall that a discrete valuation domain is a valuation domain whose value group is discrete (not necessarily of rank $1$, see \cite[Chapter VI, (A) p.48]{ZS2}). Thus, by Proposition \ref{Vpalpha}, $W_{\alpha}$ is a discrete valuation domain of $K(X)$ of rank $1$ or $2$ (notice that $W_{\alpha}/Q_{\alpha}$ is a DVR, see also Remark \ref{explicit representation}) and the residue field of $W_{\alpha}$ is a finite extension of the residue field of $V$. In addition, $\pi W_{\alpha}=M_{\alpha}^e$, where $e$ is the ramification index of $\widehat{P}_{\alpha}$ over $\widehat{P}$. In fact, since we clearly have  $M_{\alpha}^n=\{\varphi\in K(X) \mid \varphi(\alpha)\in \widehat{P}_{\alpha}^n\}$ for each $n\geq 1$, it follows that $\pi \V_{\alpha}=\widehat{P}_{\alpha}^e\Leftrightarrow\pi\in \widehat{P}_{\alpha}^e\setminus \widehat{P}_{\alpha}^{e+1}\Leftrightarrow\pi\in M_{\alpha}^e\setminus M_{\alpha}^{e+1}$.

\begin{Rem}\label{explicit representation}
By Proposition \ref{Vpalpha} and the Hasse Existence Theorem (\cite[Chapter 6, Theorem 4]{Rib}), for each pair of positive integers $e,f\geq 1$, there exists a valuation domain $W=W_{\alpha}$ of $K(X)$ lying over $V$, where $\alpha\in\abK$, whose residue field has degree $f$ over $V/P$ and $\pi W=M^e$.  In fact, by the aforementioned result of Hasse, there exists an algebraic separable extension $\K(\alpha)$ of $\K$ with ramification index $e$ and residue field degree $f$, where $\alpha\in\abK$ (i.e., a primitive element). Now, $W_{\alpha}$ is the pullback of $\V_{\alpha}$ via the evaluation morphism ${\rm ev}_{\alpha}$ and we can apply Proposition \ref{Vpalpha} to get that $W_{\alpha}$ has residue field degree equal to $f$ and $\pi W_{\alpha}=M_{\alpha}^e$. 

Given $\alpha\in\abK$, it is not difficult to give an explicit representation of the associated valuation $w_{\alpha}:K(X)^*\to\Gamma_{W_{\alpha}}$, where $\Gamma_{W_{\alpha}}$ is the corresponding value group. If the rank of $W_{\alpha}$ is $1$ then, as we saw in  Proposition \ref{Vpalpha}, $\V_{\alpha}$ is immediate over $K(X)$ (via the embedding ${\rm ev}_{\alpha}$), so:
$$w_{\alpha}(\varphi)=v_{\alpha}(\varphi(\alpha)),\;\;\forall \varphi\in K(X)^*$$ 
In particular, note that $\Gamma_{W_{\alpha}}=\Gamma_{\V_{\alpha}}$, the value group of $\V_{\alpha}$.

Suppose now the rank of $W_{\alpha}$ is $2$ and let $q\in K[X]$ be the minimal polynomial of $\alpha$. By \cite[Chapt. VI, \S 10, Thm. 17 \& p. 48]{ZS2}, the value group of $W_{\alpha}$ is order-isomorphic to $\Gamma_{q}\times\Gamma_{W_{\alpha}/Q_{\alpha}}$, where $Q_{\alpha}$ is the height one prime ideal of $W_{\alpha}$, $\Gamma_{q}$ is the value group of the DVR $(W_{\alpha})_{Q_{\alpha}}=K[X]_{(q)}$ and $\Gamma_{W_{\alpha}/Q_{\alpha}}$ is the value group of $W_{\alpha}/Q_{\alpha}$. By the proof of Proposition \ref{Vpalpha}, $\V_{\alpha}$ contains $W_{\alpha}/Q_{\alpha}$ and is immediate over it, so, in particular, $\Gamma_{W_{\alpha}/Q_{\alpha}}=\Gamma_{\V_{\alpha}}$. Given $\varphi\in K(X)^*$ there exist $k\in\Z$ and $g,h\in K[X]$, coprime with $q(X)$, such that $\varphi(X)=q(X)^k\cdot \frac{g(X)}{h(X)}$. Then
$$w_{\alpha}(\varphi(X))=(k,v_{\alpha}\left(\frac{g(\alpha)}{h(\alpha)}\right))\in \Gamma_{q}\times\Gamma_{W_{\alpha}/Q_{\alpha}}$$ 
\end{Rem}
\vskip0.5cm
Under the current assumption that $V$ is a DVR, we show in Theorem \ref{description W} that the valuation domains of Proposition \ref{Vpalpha} are the only valuation domains $W$ of $K(X)$ lying over $V$ whose residue field degree (over $V/P$) is finite and such that $\pi W=M^e$, for some $e\geq 1$. Moreover, the valuation domains $W_{\alpha}$, $\alpha\in\abV$, are precisely the unitary valuation overrings of a particular class of rings of integer-valued polynomials which we now recall (a valuation domain $W$ of $K(X)$ lying over $V$ is called unitary if its center on $V$ is the maximal ideal $P$). As in the introduction, for a finite field extension $F$ of $K$, we denote by $V_F$ the integral closure of $V$ in $F$, which is a Dedekind domain. Given a non-zero prime ideal $\mathcal{P}$ of $V_F$, which necessarily lie over $P$, we denote by $V_{F,\mathcal P}$ the localization of $V_F$ at $\mathcal{P}$.  We define also the following ring of integer-valued polynomials:
$$\Int_{K}(V_{F,\mathcal{P}})\doteqdot\{f\in K[X] \mid f(V_{F,\mathcal P})\subseteq V_{F,\mathcal P}\}$$
Let $\widehat{V_{F,\mathcal{P}}}$ be the completion of $V_{F,\mathcal{P}}$. We will use the well-known fact that
\begin{equation}\label{continuity}
\Int_{K}(V_{F,\mathcal{P}})=\Int_{K}(\widehat{V_{F,\mathcal{P}}})=\{f\in K[X] \mid f(\widehat{V_{F,\mathcal{P}}})\subseteq\widehat{V_{F,\mathcal{P}}}\}
\end{equation}
which is based on the continuity of the polynomials with respect to the $v$-adic topology. 

Before giving the main result of this section, we need the following result, which may be well-known, but for the sake of reader we give a complete proof, which is an adaption of the argument given in \cite[Chapter 6, Theorem 1, p. 151]{Rib}. 

\begin{Lem}\label{finite extension complete DVRs}
Let $V\subseteq W$ be complete DVRs with quotient fields $K\subseteq F$ and maximal ideals $P,M$, respectively. Suppose that the residue field degree $[W/M:V/P]$ is finite, equal to a positive integer $f$, and let $e$ be the ramification index of $W$ over $V$. Then $[F:K]=ef$ (so, in particular, $F/K$ is a finite extension).
\end{Lem}
\begin{proof}  

Let $\pi,\lambda$ be uniformizers of $V$ and $W$, respectively. If $e$ is the ramification index of $W$ over $V$ (which is finite, since both valuation domains are DVRs), then $\pi=\lambda^e\cdot u$, for some $u\in W^*$. Let $y_1,\ldots,y_f\in W$ be such that their residues modulo $M$ form a $V/P$-basis of $W/M$. It is well-known that the elements of the set $\{\lambda^r y_j \mid r=0,\ldots,e-1,j=1,\ldots,f\}$ are linearly independent over $K$ (for example, see first part of the proofs of \cite[Chapt VI, \S8, 1., Lemma 2]{Bourb} or \cite[Chapter 4, F., p. 114]{Rib} (or also 13.9 of Endler's book; note that for this result we don't need the completeness assumption).

For each $n\in\Z$, we set $n=q e+r$, for some $q\in\Z$ and $0\leq r<e$. We set $\sigma_n=\pi^{q}\lambda^{r}\in W$; note that $\sigma_n$ has value $n$, for each $n\in\Z$. Let now $z\in F$. There exist $n_0\in\Z\cong\Gamma_{W}$ and $u_0\in W^*$ such that $z=\sigma_{n_0}u_0$. Since $u_0$ is a unit, there exist $a_{1,0},\ldots,a_{f,0}\in V$ such that $u_0-\sum_{j=1}^f a_{j,0} y_j\in M$. Hence, $z_1=z-(\sum_{j=1}^f a_{j,0} y_j)\sigma_{n_0}$ has value strictly greater than $n_0$, so we may write
$$z=(\sum_{j=1}^f a_{j,0} y_j)\sigma_{n_0}+\sigma_{n_1}u_1$$
for some $n_1>n_0$ and $u_1\in W^*$. If we continue in this way, taking into account that $F$ is $M$-adically complete, we have the following representation for $z$ as a convergent power series:
$$z=\sum_{n\in\N}(\sum_{j=1}^f a_{j,n} y_j)\sigma_{n}$$
Using the definition of $\sigma_n$, we have that
\begin{equation}\label{representation}
z=\sum_{\substack{0\leq r<e\\1\leq j\leq f}}(\sum_{q\in\N} a_{j,n} \pi^{q})\lambda^{r}y_j
\end{equation}
Since $K$ is $P$-adically complete, for each $j$ the series $\sum_{q\in\N} a_{j,n} \pi^{q}$ is convergent, thus an element of $K$. Hence, (\ref{representation}) shows that the elements $\lambda^{r}y_j$, for $0\leq r<e$ and $1\leq j\leq f$, are a basis of $F$ over $K$.\end{proof}

In particular, the above Lemma shows that if $K$ is complete, then there are no DVRs in $K(X)$ above $V$ which have finite residue field degree, in contrast with the case of when $K$ is not complete: if $\alpha\in\K$ is transcendental over $K$, then $W_{\alpha}$ is a DVR of $K(X)$ above $V$ which has finite residue field degree (Proposition \ref{Vpalpha}).

\begin{Thm}\label{description W}
Let $W$ be a valuation domain of $K(X)$ with maximal ideal $M$, such that $W$ lies above $V$. Suppose that 
\begin{itemize}
\item[i)] $[W/M:V/P]=f$,
\item[ii)] $\pi W=M^e$,
\end{itemize}
for some $f,e\geq 1$. Then there exists $\alpha\in\PL(\abK)$ such that $\K(\alpha)/\K$ has ramification index $e$ and residue field degree $f$ and $W=W_{\alpha}$.  In particular, $X\in W\Leftrightarrow\alpha\in \abV\Leftrightarrow W$ is a valuation overring of $\Int_{K}(V_{F,\mathcal{P}})$, for some finite field extension $F$ of $K$ and prime ideal $\mcP\subset V_F$.
\end{Thm}
In particular, note that $\alpha$ has degree $e\cdot f$ over $\K$ (\cite[Chapter 6, Theorem 1]{Rib}). Also, in the case $W=W_{\infty}$ we necessarily have $e=f=1$ (see also Proposition \ref{Vpalpha}).

The existence of such a valuation domain $W$ is guaranteed also by a more general theorem given by Kuhlmann \cite[Theorem 1.4]{Kuhlmann}, but in the present context we have an explicit description of such valuation domains, see Remark \ref{explicit representation}.
\begin{proof}
Because of condition ii), the ideal $M$ is not idempotent, so $M=\varphi W$, for some $\varphi\in W$. Moreover, the rank of $W$ is $1$ or $2$; we distinguish now the two cases.

Suppose first that $W$ has rank $1$, so that $W$ is a DVR of $K(X)$. We consider the completion $\W$ of $W$ with respect to the $M$-adic topology. It is well-known that $\W$  is a DVR with field of fractions $\widehat{K(X)}$, the completion of $K(X)$ with respect to the $M$-adic topology, and $\varphi$ is a uniformizer of $\W$. Since $W$ lies over $V$, it follows that $\K$ embeds into $\widehat{K(X)}$ and $\W$ lies over $\V$. The residue field of $\W$ is isomorphic to the residue field of $W$ (\cite[Chapt. VI, \S 5, n. 3, Proposition 5]{Bourb}) and $\varphi^e \W=\pi\W$, since $M\W$ is the maximal ideal of $\W$. In particular, the same assumptions i) and ii) above for $W$ hold for its completion $\W$, so the ramification index $e(\W|\V)$ is equal to $e$ and  the residue field degree $f(\W|\V)$ is equal to $f$.  By Lemma  \ref{finite extension complete DVRs}, $\K\subseteq\widehat{K(X)}$ is a finite extension of degree $ef$.   Therefore, we have a $K$-embedding $\Phi:K(X)\hookrightarrow\widehat{K(X)}$ such that $\alpha=\Phi(X)$ is algebraic over $\K$, so, without loss of generality, we may consider $\alpha$ as an element of $\abK$;  note that the embedding $\Phi$ is nothing else that the evaluation of $X$ at $\alpha$. It follows that $\K(\alpha)$ is a finite extension of $\K$, hence complete, and since via the $K$-embedding $\Phi$ we have the containments $K(X)\subset\K(\alpha)\subseteq\widehat{K(X)}$, it follows that $\K(\alpha)$ is equal to the completion $\widehat{K(X)}$ of $K(X)$. Note that $\widehat{W}$ is isomorphic to the local ring $\V_{\alpha}$ of $\K(\alpha)$, so via the embedding $\Phi$ we have: 
$$W=\{\psi\in K(X) \mid \Phi(\psi(X))=\psi(\alpha)\in \V_{\alpha}\}=W_{\alpha}$$
Since $W$ has rank $1$ it follows that $\alpha$ is transcendental over $K$, by Proposition \ref{Vpalpha}.

Suppose now that $W\subset K(X)$ is a discrete valuation domain of rank $2$. It is known that the height one prime ideal $Q$ of $W$ is equal to $\bigcap_{n\in\N}M^n$ (\cite[Theorem 17.3]{Gilm}). Let $\pi_{Q}:W\to W/Q$ be the canonical residue map. Since $Q\cap V$ cannot be equal to $P$  because of condition ii) (i.e., $\pi\notin M^{e+1}$), we have $Q\cap V=(0)$. Hence, the restriction of $\pi_Q$ to $V$ is the identity, so $V\subseteq W/Q$ and $K$ is contained in  the quotient field $F$ of the valuation domain $W/Q$; in particular, $W/Q$ lies over $V$. Moreover, $W/Q$ is DVR with maximal ideal $M/Q$, which is generated by the residue class of $\varphi$ modulo $Q$. The localization $W_Q$  is a DVR of $K(X)$ with maximal ideal $Q$ and the residue map $W_Q\to W_Q/Q$ coincides with $\pi_Q$ over $W$ (note that the two homomorphisms have the same kernel). Now, $X^{-1}\in Q\Leftrightarrow W_Q=K[X^{-1}]_{(X^{-1})}$,  and in this case the homomorphism $W_Q\to W_Q/Q\cong K=F$ is easily seen to be ${\rm ev}_{\infty}$, the evaluation of $X$ at $\infty$. If instead $X\in W_Q$, then $W_Q=K[X]_{(q)}$, for some irreducible polynomial $q\in K[X]$.  Therefore, the residue map $\pi_Q$ is equal to the restriction to $W$ of the evaluation map ${\rm ev}_{\alpha}:K[X]_{(q)}\to K[X]_{(q)}/Q=K(\alpha)$, where $\alpha$ is the residue class of $X$ modulo $Q$, and $K(\alpha)=F$ is a finite extension of $K$. In either case ($X^{-1}\in Q\Leftrightarrow\alpha=\infty$ or $X\in W_Q$), since $W/Q$ is a DVR of $F$ containing $V_F$, it follows that $W/Q$ is equal to $V_{F,\mathcal P}$, for some prime ideal $\mathcal P\subset V_F$ (if $\alpha=\infty$, then $V_{F,\mathcal{P}}=V$). We identify $\alpha\in F$ with its image in the $\mathcal{P}$-adic completion $\widehat{F}_{\mathcal{P}}=\K(\alpha)\subset\abK$ of $F$ (see \cite[Chapt. 4, L., p. 121]{Rib}); in this way $\alpha$ is uniquely associated to an element of $\PL(\abK)$, which is $\infty$ exactly in the case $X^{-1}\in Q$. Since $V_{F,\mathcal{P}}=\abV\cap F$ and $\alpha\in F$, we have:
$$W=\{\psi\in K(X) \mid \psi(\alpha)\in V_{F,\mathcal P}\}=W_{\alpha}.$$
In fact, the value at $\alpha$ of $\psi\in K(X)$ is in $V_{F,\mathcal P}=W/Q=\pi_{Q}(W)$ if and only if $\psi$ is in $W,$ since $\pi_Q=({\rm ev}_{\alpha})_{|W}$ and the homomorphisms ${\rm ev}_{\alpha}$ and $\pi_Q$ share the same kernel, namely the ideal $Q\subset W$. Note that the completion of $V_{F,\mathcal{P}}$ is isomorphic to $\V_{\alpha}$. Since $W/M\cong(W/Q)/(M/Q)$, which is the residue field of $W/Q=V_{F,\mathcal{P}}$,  the residue field degree of $\K(\alpha)$ over $\K$ is $f$. From the assumption $\pi W=M^e$, it follows that $\pi\cdot W/Q=(M/Q)^e$, so the ramification index of $\K(\alpha)/\K$ is $e$.

We prove now the last equivalences, whether $W=W_{\alpha}$ has rank $1$ or $2$ (note that in these cases $\alpha\not=\infty$). Suppose that $X\in W$. If $W_{\alpha}$ has rank $1$, then $\Phi(X)=\alpha\in\W=\V_{\alpha}$, so $\alpha$ is integral over $\V$. If $W_{\alpha}$ has rank $2$, then $\pi_Q(X)=\alpha\in W/Q=V_{F,\mcP}$, and so $\alpha$ is integral over $\V$ (in this case, $\alpha$ is not necessarily integral over $V$, which is the case exactly when there is just one prime ideal $\mcP$ in $V_F$ above $P$). 

Suppose that $\alpha$ is integral over $\V$. In case $W=W_{\alpha}$ has rank $2$, then $\alpha\in\abV\cap F=V_{F,\mathcal{P}}$, and so $\Int_K(V_{F,\mathcal{P}})\subset W_{\alpha}$. Suppose now that $W_{\alpha}$ has rank $1$ and let $\widehat q(X)$ be the minimal polynomial of $\alpha$ over $\K$, of degree $ef$. By Krasner's Lemma (\cite[Chapt. 5, G.]{Rib}), if $q\in K[X]$ is a monic polynomial of degree $ef$ which is $v$-adically sufficiently close to $\widehat q(X)$, then $q(X)$ is irreducible over $\K$, hence also over $K$. Let $F=K(\beta)$ be the number field of degree $ef$ generated by a root $\beta$ of $q(X)$. Note that there exists only one prime ideal $\mathcal{P}$ in $V_F$ above $P$ (precisely because $q(X)$ is irreducible over the completion $\K$, see \cite[Chapt. 6, B.]{Rib}); in particular, $V_F$ is a DVR and the $\mathcal P$-adic completion of $F$ is isomorphic to $\K(\alpha)=\K(\beta)$. Moreover, $e(\mathcal P|P)=e$ and $f(\mathcal P|P)=f$ (\cite[Chapt. VI, \S 5, n. 3, Proposition 5]{Bourb}). Let $W'_{\alpha}=\{\psi\in F(X) \mid \psi(\alpha)\in\V_{\alpha}\}$ be the valuation domain of $F(X)$ corresponding to $\alpha$. Since the $\mcP$-adic completion of $V_F$ is equal to $\V_{\alpha}$, by (\ref{continuity}) we have $\Int(V_F)=\Int_F(\V_{\alpha})$, and the latter ring is clearly contained in $W'_{\alpha}$, because by assumption $\alpha\in\abV\Leftrightarrow\alpha\in\V_{\alpha}$. Contracting everything down to $K(X)$ we get $\Int_K(V_F)\subset W_{\alpha}$. 
 
Finally, if $W_{\alpha}$ contains $\Int_K(V_{F,\mathcal{P}})\supset V[X]$, then $X$ is in $W_{\alpha}$. The proof of the last claim of the statement is now complete.
\end{proof}

\begin{Rem}\label{Kap}
Let $W_{\alpha}$ be a valuation domain of $K(X)$, with $\alpha\in\abK$. We have seen in the proof of Theorem \ref{description W} that the completion of $W_{\alpha}$ with respect to the $M_{\alpha}$-adic topology is isomorphic to $\V_{\alpha}$: in fact, if $\alpha$ is transcendental over $K$, this is clear. If $\alpha$ is algebraic over $K$, then the $M_{\alpha}$-adic completion of $W_{\alpha}$ is equal to the $(M_{\alpha}/Q_{\alpha})$-adic completion of $W_{\alpha}/Q_{\alpha}$, where $Q_{\alpha}$ is the height one prime ideal of $W_{\alpha}$, since $Q_{\alpha}=\bigcap_{n\in\N}M_{\alpha}^n$. It follows that if the rank of $W_{\alpha}$ is $1$, then $W_{\alpha}$ is immediate over $K$ if and only if $e=f=1$, thus $\alpha\in \K$, so $\alpha$ is the limit of a (pseudo-)Cauchy sequence $\{\alpha_k\}_{k\in\N}\subset K$ of transcendental type, according to the terminology used by Kaplansky in his paper (see \cite[Theorem 2]{Kaplansky}). Similarly, if the rank of $W_{\alpha}$ is $2$, then $W_{\alpha}/Q_{\alpha}$ is immediate over $K$ if and only if $e=f=1$, thus $\alpha\in \K$. In this latter case, the corresponding Cauchy sequence $\{\alpha_k\}_{k\in\N}$ is of algebraic type.
\end{Rem}

\vskip0.2cm
\begin{Rem}\label{MacLane valuations}
We show now how the valuation domains $W_{\alpha}\subset K(X)$, $\alpha\in\PL(\abK)$, are related with the work of MacLane on valuations of the rational function field $K(X)$  which extend a given DVR $V$ of $K$ (see \cite{MacLane,MacLaneDuke}). The class of valuation domains of Definition \ref{Def Vpalpha} are exactly the \emph{constant degree limit valuations} considered by MacLane in \cite[\S 7. p. 375]{MacLane} (we refer to that paper for all the unexplained terminology that follows). Such a valuation domain is obtained as a suitable limit of the so-called inductive commensurable valuations where the degree of the associated sequence of key polynomials is bounded. These kind of valuation domains can be of two types, finite limit valuations or infinite limit valuations (see \cite[\S 6 \& \S 7, pp. 372-377, \& Theorem 7.1]{MacLane} and \cite[Note, p. 108]{LopTart}). The first type of valuation domains are DVRs of $K(X)$ with residue field which is a finite extension of $V/P$ (\cite[Theorem 7.1 \& Theorem 14.1]{MacLane}), so by Theorem \ref{description W}  and Proposition \ref{Vpalpha} they correspond to the $W_{\alpha}$'s where $\alpha\in\abK$ is transcendental over $K$. The second type of valuation domain is treated by MacLane as a one-dimensional valuation domain with the value group extended by adding $\infty$ (they are also called \emph{pseudo-valuations} in \cite{NartCo}). In fact, as noted in \cite[Lemma 1.23 \& p. 109]{LopTart}, these last kind of limit valuations are $2$-dimensional discrete valuation domains. Moreover, the one-dimensional valuation overring of such a valuation domain $W$ is of the form $K[X]_{(q)}$, for some irreducible polynomial $q\in K[X]$ and the residue field of $W$ is a finite extension of the residue field $V/P$ (this can be verified directly or also by a suitable modification of the original proof by MacLane in the case of finite limit valuations, \cite[Theorem 14.1]{MacLane}). Therefore, these valuations $W$ are exactly those of the form $W_{\alpha}$, for $\alpha\in \abK$ which is algebraic over $K$. 

Conversely, suppose we have a valuation domain $W_{\alpha}$, $\alpha\in\abK$. By \cite[Theorem 8.1]{MacLane} $W_{\alpha}$ can be realized as a limit valuation since its residue field is finite algebraic over $V/P$ and the residue field of an inductive commensurable valuation is a transcendental extension of $V/P$ (\cite[Theorem 12.1]{MacLane}). Moreover, by \cite[Theorem 14.1]{MacLane} the degree of the key polynomials is necessarily bounded, so $W_{\alpha}$ can be realized as a constant degree limit valuation. 
\end{Rem}
\vskip0.2cm
\begin{Rem}\label{Unitary val dom LopWer}
In \cite{LopWer} Loper and Werner construct   Pr\"ufer domains contained between $\Z[X]$ and $\Q[X]$ by considering arbitrary intersections of suitable valuation domains of $\Q(X)$, in order to obtain examples of Pr\"ufer domains properly  contained in $\Int(\Z)=\{f\in\Q[X] \mid f(\Z)\subseteq\Z\}$, the classical ring of integer-valued polynomials over $\Z$ (see \cite[Construction 2.3 \& Corollary 2.12]{LopWer}). The valuation domains used in that construction are exactly those introduced in Definition \ref{Def Vpalpha} and we show now that the Pr\"ufer domains of \cite{LopWer} can be represented as rings of integer-valued polynomials. Indeed, for each prime $p\in\Z$, the valuation domains $V_i$ of \cite[Construction 2.3]{LopWer} have finite residue field of cardinality bounded by a prescribed positive integer $f_p$ and maximal ideal $M_i$ such that $pV_i=M_i^{e_i}$, for some $e_i$ bounded by a prescribed positive integer $e_p$. Hence, by Theorem \ref{description W}, each of these valuation domains is equal to $W_{p,\alpha}=\{\varphi\in\Q(X) \mid \varphi(\alpha)\in\overline{\Z_p}\}$, for some $\alpha$ in the absolute integral closure $\overline{\Z_p}$ of the ring $\Z_p$ of $p$-adic integers, whose degree over $\Q_p$ is bounded by $n_p=e_p\cdot f_p$. Let $\Omega_p\subset\overline{\Z_p}$ be the set of all such elements $\alpha$. Then we have
\begin{equation*}
D_p\doteqdot \bigcap_{\alpha\in\Omega_p}W_{p,\alpha}\cap\Q[X]=\Int_{\Q}(\Omega_p,\overline{\Z_p})=\{f\in\Q[X] \mid f(\Omega_p)\subseteq\overline{\Z_p}\}
\end{equation*}
which is the ring of polynomials with rational coefficients which are integer-valued over the set $\Omega_p$ with respect to $\Z_p$. The ring $D$ obtained in \cite[Construction 2.3]{LopWer} as the intersection of all the rings $\{D_p \mid p\in\Z \textnormal{ prime}\}$, is thus represented as an intersection of such rings of integer-valued polynomials over different subsets of integral elements over $\Z_p$ of bounded degree, as $p$ ranges through the primes of $\Z$. Following the notation of \cite{ChabPer}, we can give a more concise representation of $D$. Let
$$\underline{\Omega}\doteqdot\prod_{p\in\mathbb{P}}\Omega_p\subseteq\prod_{p\in\mathbb{P}}\overline{\Z_p}\doteqdot\overline{\widehat{\Z}}$$
then
\begin{equation}\label{2 representation}
D=\Int_{\Q}(\underline{\Omega},\overline{\widehat{\Z}})=\{f\in\Q[X] \mid f(\underline{\alpha})\in \overline{\widehat{\Z}},\;\;\forall\underline{\alpha}\in\underline{\Omega}\}
\end{equation}
where, for $f\in\Q[X]$ and $\underline{\alpha}=(\alpha_p)_p\in \overline{\widehat{\Z}}$, we set $f(\underline{\alpha})\doteqdot(f(\alpha_p))_p\in\prod_{p\in\mathbb{P}}\overline{\Q_p}$.
\end{Rem}
\vskip0.4cm
\section{An ultrametric space of valuation domains of $K(X)$}

In this section, we recover our initial assumptions, thus $V$ is a valuation domain of rank $1$.  Throughout this section, we denote by $\mathfrak{W}$ the set of all valuation domains $W_{\alpha}$ of $K(X)$, as $\alpha$ ranges in $\abK$. For each $\alpha\in\abK$, we also set:
$$\mcW_{\alpha}\doteqdot\{\psi\in\K(X) \mid \psi(\alpha)\in\abV\}$$
By Proposition \ref{Vpalpha}, $\mcW_{\alpha}$ is a valuation domain of $\K(X)$ of rank $2$, since $\alpha$ is algebraic over $\K$ by definition. Clearly, $\mcW_{\alpha}\cap K(X)=W_{\alpha}$, and  $\mcW_{\alpha}=W_{\alpha}$ if $K$ is $v$-adically complete. Moreover, note that $\mcW_{\alpha}$ is immediate over $K(X)$ if and only if $\alpha$ is algebraic over $K$ (if $\alpha$ is transcendental over $K$, then $W_{\alpha}$ has rank $1$, by Proposition \ref{Vpalpha}). We denote by $\widehat{\mathfrak{W}}$ the set of all valuation domains $\mcW_{\alpha}$ of $\K(X)$, for $\alpha\in\abK$. We also denote by  $\Pirr=\Pirr(\K)$ the set of the monic irreducible polynomials over $\K$. Given $\alpha\in\abK$, we denote by $p_{\alpha}\in\Pirr$ the minimal polynomial of $\alpha$ over $\K$. Given $p\in\Pirr$, we let $\Omega_p\subset\abK$ be the set of roots of $p(X)$. Let  $|\;\;|$ be  the absolute value $v$ induces on $\abK$. In general, given an ultrametric space $(S,|\;|)$, $s\in S$ and $r\in\R$, $r>0$, we let $B(s,r)=\{s'\in S \mid |s-s'|<r\}$ be the open ball of center $s$ and radius $r$.

In this section we show that there is a one-to-one correspondence from $\mathfrak{W}$  to the set of orbits of $\abK$ under the action of the absolute Galois group $G_{\K}=\Gal(\abK/\K)$, that is, given $\alpha,\alpha'\in\abK$,  $W_{\alpha}=W_{\alpha'}$ if and only if $\alpha$ and $\alpha'$ are roots of the same irreducible polynomial over $\K$. Equivalently, the set $\mathfrak{W}$ is in bijection with the set $\Pirr$. Similarly, $\widehat{\mathfrak{W}}$ is in bijection with $\Pirr$. We introduce now in these spaces suitable natural topologies. We endow $\mathfrak{W}$ with the Zariski topology, that is, the topology which has as an open basis the sets $E(\varphi_1,\ldots,\varphi_n)=\{W_{\alpha} \mid \varphi_i\in W_{\alpha}, \forall i=1,\ldots,n\}$, for $\varphi_i\in K(X)$, $i=1,\ldots,n$, $n\in\N$ (see \cite[Chapt. VI, \S 17]{ZS2}). The set $\widehat{\mathfrak{W}}$ is endowed with a similar topology: for $\psi_i\in \K(X)$, $i=1,\ldots,n$, we set $\widehat E(\psi_1,\ldots,\psi_n)=\{\mcW_{\alpha} \mid \psi_i\in \mcW_{\alpha}, \forall i=1,\ldots,n\}$. We endow the set $\Pirr$ with the following ultrametric distance, introduced by Krasner (see \cite{Krasner}): for $p,q\in\Pirr$, we set
$$\Delta(p,q)\doteqdot\min\{|\alpha-\beta| \mid \alpha\in\Omega_p,\beta\in\Omega_q\}$$
In other words, the function $\Delta(p,q)$ measures the smallest distance between the roots of $p(X)$ and the roots of $q(X)$. As Krasner points out, for each $\alpha\in\Omega_p$ there exists $\beta\in\Omega_q$ such that $|\alpha-\beta|=\Delta(p,q)$. The other main result of this section is that with the topologies we have introduced $\Pirr$, $\mathfrak{W}$ and $\widehat{\mathfrak{W}}$ are homeomorphic.

We recall the following formula due to Krasner (see \cite[p. 150-151]{Krasner}): given $p\in\Pirr$ of degree $n$ with set of roots $\Omega_p=\{\alpha=\alpha_1,\ldots,\alpha_n\}\subset\abK$ and $\beta\in\abK$ with minimal polynomial $q\in\Pirr$, we have
\begin{equation}\label{Krasner}
|p(\beta)|=\prod_{i=1}^n\max\{\Delta(p,q),|\alpha-\alpha_i|\}
\end{equation}
 Recall that, for $\varphi\in K(X)$ and $p\in\Pirr$, the valuation of  $\varphi(\alpha)$ in $\abK$, for $\alpha\in\Omega_p$, does not depend on the choice of $\alpha$ in $\Omega_p$. In fact, for $\alpha,\alpha'\in\Omega_p$, $\alpha\not=\alpha'$,
the elements $\varphi(\alpha),\varphi(\alpha')$ are conjugated over $\K$, i.e., $\varphi(\alpha')=\sigma(\varphi(\alpha))$, for some $\sigma\in G_{\K}$; then since $\K$ is complete of rank one, $v$ and $v\circ \sigma$ coincide (\cite[A., p. 127]{Rib}), thus $v(\varphi(\alpha))=v(\sigma(\varphi(\alpha)))=v(\varphi(\alpha'))$. 

The right hand side of (\ref{Krasner}), which we denote by $M_{p}(\Delta(p,q))$, is a strictly increasing function of $\Delta(p,q)$,  that is, $\Delta(p,q)<\Delta(p,q')\Leftrightarrow M_p(\Delta(p,q))<M_p(\Delta(p,q'))$.  In particular, formula (\ref{Krasner}) shows that $|p(\beta)|$ depends only on $p(X)$ and $\Delta(p,p_{\beta})$. Therefore, $|p(\beta)|=M_p(\Delta(p,p_{\beta}))=M_p(\Delta(p,q))$ for each $q\in\Pirr$ such that $\Delta(p,p_{\beta})=\Delta(p,q)$. 
More generally, the real-valued function
\begin{equation}\label{Mpr}
M_p(r)\doteqdot\prod_{i=1}^n\max\{r,|\alpha-\alpha_i|\}
\end{equation}
is a strictly increasing function of the real variable $r$.  It follows immediately that for each $r\in\mathbb{R}^+$ we have $M_p(r)=|p(\beta)|$, for each $\beta\in \abK$ such that $\Delta(p,p_{\beta})=r$.

\begin{Rem}\label{continuity rat fun}
We will use the following well-known fact: a rational function $\varphi\in K(X)$ is a continuous function over $\abK$ with respect to the $v$-adic topology on its domain of definition, that is, if $\alpha\in\abK$ is such that $\varphi(\alpha)\not=\infty$, then, for each $\varepsilon\in\R$, $\varepsilon>0$, there exists $\delta\in\R$, $\delta>0$, such that for all $\alpha'\in B(\alpha,\delta)$ we have $\varphi(\alpha')\in B(\varphi(\alpha),\varepsilon)$. In particular, if $\varphi$ is integral at $\alpha$ (i.e., $\varphi(\alpha)\in\abV$) then for $\varepsilon<1$, the corresponding $\delta$ is such that $\varphi$ is integral over $B(\alpha,\delta)$, that is, $\varphi\in W_{\alpha'}$, for all $\alpha'\in B(\alpha,\delta)$. Actually, we note that, since we are considering rational functions over $K$, if $\varphi$ is integral at $\alpha'\in\abK$ and $p(X)=p_{\alpha'}(X)\in\Pirr$, then $\varphi$ is integral over the set $\Omega_{p}$, so, it is sufficient that an element of $\Omega_{p}$ is in $B(\alpha,\delta)$ in order for $\varphi$ to be integral at $\alpha'$; equivalently, for all $p_{\alpha'}\in B(p_{\alpha},\delta)=\{p\in\Pirr \mid \Delta(p,p_{\alpha})<\delta\}$, $\varphi\in W_{\alpha'}$.
\end{Rem}
We will also consider the sets formed by the contraction to $K[X]$ and to $\K[X]$ of the valuation domains of $\mathfrak{W}$ and $\widehat{\mathfrak{W}}$, respectively:
\begin{equation*}
\mathfrak{W}_{K[X]}=\{W_{\alpha}\cap K[X] \mid \alpha\in\abK\},\;\;
\widehat{\mathfrak{W}}_{\K[X]}=\{\mcW_{\alpha}\cap \K[X] \mid \alpha\in\abK\}
\end{equation*}
These contractions were first considered by MacLane in \cite[p. 382]{MacLane}, where they are called \emph{value rings}; see also \cite{LopTart} for a deeper study of their properties. Note also that we have the equality $W_{\alpha}\cap K[X]=\{f\in K[X] \mid f(\alpha)\in\abV\}$, where the last ring is a ring of integer-valued polynomials over the finite set $\{\alpha\}$, denoted by $\Int_K(\{\alpha\},\abV)$ in \cite{PerFinite}. The set $\mathfrak{W}_{K[X]}$ becomes a topological space when it is endowed with the natural Zariski topology, where a basis is given by $E_{K[X]}(f_1,\ldots,f_n)=\{W_{\alpha}\cap K[X] \mid f_i\in W_{\alpha}\cap K[X],i=1,\ldots,n\}$, where $f_1,\ldots,f_n$ are elements of $K[X]$ and $n\in\N$. Similar topology  is given to $\widehat{\mathfrak{W}}_{\K[X]}$. We will show in our last main theorem that also $\mathfrak{W}_{K[X]}$ and $\widehat{\mathfrak{W}}_{\K[X]}$ are homeomorphic to $\Pirr$, $\mathfrak{W}$ and $\widehat{\mathfrak{W}}$. We show first in the next theorem that these sets are in bijection with each other.
\vskip0.5cm
\begin{Thm}\label{equivalent conditions}
Let $\alpha_1,\alpha_2\in\abK$. Then the following conditions are equivalent:
\begin{itemize}
\item[i)] $W_{\alpha_1}=W_{\alpha_2}$.
\item[ii)] $W_{\alpha_1}\cap K[X]=W_{\alpha_2}\cap K[X]$.
\item[iii)] $\mcW_{\alpha_1}\cap \K[X]=\mcW_{\alpha_2}\cap \K[X]$.
\item[iv)] $\alpha_1,\alpha_2$ are conjugated over $\K$.
\item[v)] $\mcW_{\alpha_1}=\mcW_{\alpha_2}$.
\end{itemize}
In particular, the sets $\mathfrak{W}$, $\widehat{\mathfrak{W}}$, $\mathfrak{W}_{K[X]}$, $\widehat{\mathfrak{W}}_{\K[X]}$ and $\Pirr$ are in bijective correspondence.
\end{Thm}
\begin{proof} 
$i)\Rightarrow ii)$. Clear.

 $ii)\Rightarrow iii)$. Suppose $W_{\alpha_1}\cap K[X]=W_{\alpha_2}\cap K[X]$ and let $f\in \mcW_{\alpha_1}\cap \K[X]$. Let us write $f(X)=\sum_{i=0}^h \widehat{a}_i X^i$, with $\widehat{a}_i\in\K$, $0\leq i\leq h$ and let us consider $a_i\in K$, $0\leq i\leq h$, such that $v((\widehat{a}_i-a_i)\alpha_j^i)=v(\widehat{a}_i-a_i)+iv(\alpha_j)>v(f(\alpha_1))$, for all $0\leq i\leq h$ and $j=1,2$. If $g(X)=\sum_{i=0}^h a_i X^i\in K[X]$, then $v(g(\alpha_1))\geq \min\{v((g-f)(\alpha_1)),v(f(\alpha_1))\}=v(f(\alpha_1))\geq 0$ (note that $v((g-f)(\alpha_1))\geq\min\{v(\widehat{a}_i-a_i)+iv(\alpha_1) | i\in\{0,\ldots,h\}\}$). Thus, $g\in W_{\alpha_1}\cap K[X]=W_{\alpha_2}\cap K[X]$ and $v(f(\alpha_2))\geq \min\{v((f-g)(\alpha_2)),v(g(\alpha_2))\}\geq 0$ (note that $v((f-g)(\alpha_2))\geq\min\{v(\widehat{a}_i-a_i)+iv(\alpha_2) | i\in\{0,\ldots,h\}\}>v(f(\alpha_1))$). Hence, $f\in\mcW_{\alpha_2}\cap \K[X]$. Since $f(X)$ was arbitrary, this shows that $\mcW_{\alpha_1}\cap \K[X]\subseteq\mcW_{\alpha_2}\cap \K[X]$ and the other inclusion is proved in the same way.

$iii)\Rightarrow iv)$. Suppose $\mcW_{\alpha_1}\cap \K[X]=\mcW_{\alpha_2}\cap \K[X]$ and let $p=p_{\alpha_1}\in\Pirr$. Let us fix $\omega\in P$, $\omega\not=0$. If $p(\alpha_2)\not=0$, then there exists $n\in\N$ such that $\frac{p(\alpha_2)}{\omega^n}\notin\abV$, which is a contradiction, since for every $n\in\N$ we have $\frac{p(X)}{\omega^n}\in \mcW_{\alpha_1}\cap \K[X]=\mcW_{\alpha_2}\cap \K[X]$. Therefore $p(\alpha_2)=0$, so that $\alpha_1,\alpha_2$ are conjugated over $\K$.

$iv)\Rightarrow v)$. Suppose there exists $\sigma\in G_{\K}$ such that $\sigma(\alpha_1)=\alpha_2$. Given $f\in \mcW_{\alpha_1}$, $\sigma(f)=f\in \mcW_{\alpha_2}$, so $\mcW_{\alpha_1}\subseteq \mcW_{\alpha_2}$. The other inclusion is proved symmetrically, so $\mcW_{\alpha_1}= \mcW_{\alpha_2}$.

$v)\Rightarrow i)$. Since $\mcW_{\alpha_i}\cap K(X)=W_{\alpha_i}$, $i=1,2$, the claim follows immediately.

The last statement is now clear.
\end{proof}
\vskip0.3cm
\begin{Prop}\label{equivalent subbasis}
Let $\alpha\in\abK$, $r>0$ and $\omega\in P$, $\omega\not=0$. Then there exist $q\in K[X]$ and $n\in\N$ such that $\frac{q(X)}{\omega^n}$ is integral at $\alpha$ and for all $W_{\alpha'}\in E(\frac{q(X)}{\omega^n})$, we have $p_{\alpha'}\in B(p_{\alpha},r)$. In particular, the family $\{E(\frac{q(X)}{\omega^n}) \mid q\in K[X], n\in\N\}$ is  a subbasis for the Zariski topology on $\mathfrak{W}$. 
\end{Prop}
\begin{proof}
Given $\alpha\in\abK$, let $p=p_{\alpha}\in\Pirr$ and $d$ the degree of $p(X)$. Let $B=B(p,r)$, where $r$ is any given positive real number. We suppose first that $p(X)$ is separable.  We choose $n\in\N$ such that $|\omega^n|< M_{p}(r)$. Let $\Delta(p)$ be the minimum distance of the distinct roots $\alpha=\alpha_1,\ldots,\alpha_d$ of $p$ in $\abK$. Then there exists $\delta>0$ such that if $q\in K[X]$ is monic of degree $d$ and $|q-p|_G< \delta$ (where $|\;\;|_G$ is the absolute value associated with the Gauss valuation $v_G$), then for each $\alpha_i\in\Omega_{p}$, $i=1,\ldots,d$, there exists a unique root $\beta_i\in\Omega_q$ such that $|\alpha_i-\beta_i|<\min\{\Delta(p),r\}$ (see \cite[Chapt. 5, p. 139]{Rib}). In particular, $\Delta(q,p)<r$  and $q(X)$ is irreducible over $\K$ (hence also over $K$; this holds by Krasner's Lemma, see \cite[Chapt. 5, G. p. 139]{Rib}). Moreover, up to a choice of a smaller $\delta$, we may also suppose that $|q(\alpha)|\leq |\omega^n|$, so that the polynomial $\frac{q(X)}{\omega^n}$ is integral at $\alpha$. We claim now that $M_{p}(\rho)=M_{q}(\rho)$ for each $\rho\in\R$, and by the very definition (\ref{Mpr}) it is sufficient to show that $|\alpha_1-\alpha_i|=|\beta_1-\beta_i|$, for each $i=2,\ldots,d$. Indeed, we have $$|\beta_1-\beta_i|=|\beta_1-\alpha_1+\alpha_1-\alpha_i+\alpha_i-\beta_i|=|\alpha_1-\alpha_i|$$
 since $|\alpha_1-\alpha_i|\geq \Delta(p)>|\beta_j-\alpha_j|$ for $j=1,i$. Finally, we have 
$$W_{\alpha'}\in E\left(\frac{q(X)}{\omega^n}\right)\Leftrightarrow |q(\alpha')|=M_q(\Delta(q,p_{\alpha'}))=M_p(\Delta(q,p_{\alpha'}))\leq|\omega^n|< M_p(r)\Rightarrow \Delta(q,p_{\alpha'})< r$$
since $M_p(\cdot)$ is a strictly increasing function. Therefore, $\Delta(p_{\alpha'},p)\leq\max\{\Delta(p_{\alpha'},q),\Delta(q,p)\}<r$, thus, $p_{\alpha'}\in B$.

We suppose now that $p(X)$ is inseparable. Let $l>0$ be the characteristic of $K$ and let $p(X)=\widetilde{p}(X^{l^m})$, where $m\geq1$ and $\widetilde{p}\in\Pirr$ is separable. Note that if $\Omega_p=\{\alpha=\alpha_1,\ldots,\alpha_t\}$ (where $t<d$ is the number of distinct roots of $p(X)$), then $\Omega_{\widetilde{p}}=\{\alpha_1^{l^m},\ldots,\alpha_t^{l^m}\}$. We set $\gamma=\alpha^{l^m}$ and $\gamma_i=\alpha_i^{l^m}$, for $i=2,\ldots,t$. Let $\widetilde{r}=r^{l^m}$. By the first part of the proof, there exist $\widetilde{q}\in\Pirr\cap K[X]$ and $n\in\N$ such that $\frac{\widetilde{q}(X)}{\omega^n}$ is integral at $\gamma=\alpha^{l^m}$ and for each $W_{\gamma'}\in E(\frac{\widetilde{q}(X)}{\omega^n})$ we have $p_{\gamma'}\in B(\widetilde{p},\widetilde{r})$. We set $q(X)=\widetilde{q}(X^{l^m})$. We have that $\frac{q(X)}{\omega^n}$ is integral at $\alpha$ since $q(\alpha)=\widetilde{q}(\gamma)$. Moreover, let $W_{\alpha'}\in E(\frac{q(X)}{\omega^n})$. This implies that $W_{\gamma'}\in E(\frac{\widetilde{q}(X)}{\omega^n})$, where $\gamma'=\alpha'^{l^m}$. Therefore, $\Delta(p_{\gamma'},\widetilde{p})<\widetilde{r}$. Now, 
\begin{align*}
\Delta(p_{\gamma'},\widetilde{p})=\min\{|\gamma'-\gamma_i| \mid i=1,\ldots,t\}&=\\ 
=\min\{|\alpha'^{l^m}-\alpha_i^{l^m}| \mid i=1,\ldots,t\}&=\min\{|\alpha'-\alpha_i|^{l^m} \mid i=1,\ldots,t\}=\Delta(p_{\alpha'},p)^{l^m}
\end{align*}
Hence, $\Delta(p_{\alpha'},p)<r$, as wanted.

For the last statement, the finite intersections of the sets $E(\frac{q(X)}{\omega^n})$, for $q\in K[X]$ and $n\in\N$, form a basis for a topology on $\mathfrak{W}$ which is weaker than the Zariski topology on $\mathfrak{W}$. Let $E(\varphi)$, $\varphi\in K(X)$ be an element of the subbasis for the Zariski topology. Let $W_{\alpha}\in E(\varphi)$, $\alpha\in\abK$. By continuity of $\varphi$, there exists a neighborhood $B=B(p_{\alpha},r)$ of $p_{\alpha}(X)$, $r>0$, such that for all $p_{\alpha'}\in B$, we have $\varphi\in W_{\alpha'}$ (see Remark \ref{continuity rat fun}). Then by what we have proved above, there exists a basic open set $E=E(\frac{q(X)}{\omega^n})$, where $q\in K[X]$ and $n\in\N$, such that $W_{\alpha}\in E\subseteq E(\varphi)$, which shows that $\{E(\frac{q(X)}{\omega^n})\mid q\in K[X],n\in\N\}$ is a subbasis for the Zarisky topology on $\mathfrak{W}$.
\end{proof}
Note that in the case $\abK$ is a separable extension of $\K$, the above proof shows that we may also suppose that the polynomials $q(X)$ above are irreducible over $\K$ (hence also over $K$). Also, the same Proposition shows that $\{\widehat{E}(\frac{q(X)}{\pi^n})\mid q\in K[X],n\in\N\}$ is a subbasis for the Zarisky topology on $\widehat{\mathfrak{W}}$. These results are the main ingredients for the proof of our last theorem.
\vskip0.3cm

\begin{Thm}\label{homeomorphism}
The topological spaces $\mathfrak{W}$, $\widehat{\mathfrak{W}}$, $\mathfrak{W}_{K[X]}$, $\widehat{\mathfrak{W}}_{\K[X]}$ and $\Pirr$ are homeomorphic.
\end{Thm}
\begin{proof}
Let $\mathcal I$ be  the bijection  from $\widehat{\mathfrak{W}}$ to $\mathfrak{W}$ defined by $\mcW_{\alpha}\mapsto W_{\alpha}$, for each $\alpha\in\abK$ (Theorem \ref{equivalent conditions}). Clearly, this map is continuous, since for $\varphi\in K(X)$ we have $\mathcal{I}^{-1}(E(\varphi))=\{\mathcal{W}_{\alpha} \mid W_{\alpha}\ni \varphi\}=\\ \{\mathcal{W}_{\alpha} \mid \mathcal{W}_{\alpha}\ni \varphi\}=\widehat{E}(\varphi)$. Conversely, let $\widehat{E}(\psi)$, $\psi\in\K(X)$, be a basic open set and let $W_{\alpha}$ be an element of $\mathcal{I}(\widehat{E}(\psi))$. Let $\omega\in P$, $\omega\not=0$. By Proposition \ref{equivalent subbasis}, there exist $q\in K[X]$ and $n\in\N$ such that $W_{\alpha}\in E(\frac{q(X)}{\omega^n})\subseteq\mathcal{I}(\widehat{E}(\psi))$. Hence, $\mathcal{I}$ is open and $\mathfrak{W}$ and $\widehat{\mathfrak{W}}$ are homeomorphic. 

The same proof shows that the maps $\mathfrak{W}\to\mathfrak{W}_{K[X]}$, $W_{\alpha}\mapsto W_{\alpha}\cap K[X]$, and $\widehat{\mathfrak{W}}\to\widehat{\mathfrak{W}}_{\K[X]}$, $\mcW_{\alpha}\mapsto \mcW_{\alpha}\cap \K[X]$, for $\alpha\in\abK$, are homeomorphisms (they are bijections by Theorem \ref{equivalent conditions}). 

Finally, we prove that $\Phi:\Pirr\to\mathfrak{W}$, $p_{\alpha}\mapsto W_{\alpha}$, $\alpha\in\abK$, is an homeomorphism.
We prove first that $\Phi$ is continuous. It is sufficient to show that, for any $\varphi\in K(X)$, the preimage via $\Phi$ of $E(\varphi)$ in $\Pirr$ is open. Let $q\in\Pirr$ be an element of this preimage, so that $\varphi\in W_{\beta}$, where $\beta$ is any root of $q(X)$. We have to show that  there exists a neighborhood of $q(X)$ in $\Pirr$ which is contained in $\Phi^{-1}(E(\varphi))$, that is, if $\Delta(p,q)$, $p\in\Pirr$, is sufficiently small, then $\Phi(p)=W_{\beta'}$ is in $E(\varphi)$, where $\beta'\in\abK$ is any root of $p(X)$.  But this holds because  a rational function $\varphi$ is a continuous function with respect to the $v$-adic topology on its domain of definition, see Remark \ref{continuity rat fun}. 

Next, we show that $\Phi$ is an open map. Let $B(p,r)$ be an open neighborhood of a given $p\in\Pirr$ for some $r\in\R$,  $r>0$, and let $W_{\alpha}\in \Phi(B(p,r))\Leftrightarrow \Delta(p_{\alpha},p)<r$. By well-known properties of ultrametric distances, $B(p,r)=B(p_{\alpha},r)$. Hence, without loss of generality, we may suppose that $p(\alpha)=0$. Let $\omega\in P$, $\omega\not=0$. By Proposition \ref{equivalent subbasis}, there exist $q\in K[X]$ and $n\in\N$ such that $W_{\alpha}\in E(\frac{q(X)}{\omega^n})\subseteq\Phi(B(p,r))$, so the map $\Phi$ is open, hence an homeomorphism.
\end{proof}

\vskip0.2cm

\subsection*{Acknowledgements}
The author warmly thanks the referee for his/her mostly valuable suggestions which allow to state the main results in greater generality. The author wishes also to thank Bruce Olberding for precious and insightful comments which improved the quality of the paper.

\end{document}